\documentclass[a4paper,11pt]{article}

  \renewcommand\appendix{\par
  \setcounter{section}{0}
  \setcounter{subsection}{0}
  \setcounter{figure}{0}
  \setcounter{table}{0}
  \renewcommand\thesection{ Appendix \Alph{section}}
  \renewcommand\thefigure{\Alph{section}\arabic{figure}}
  \renewcommand\thetable{\Alph{section}\arabic{table}}
}

 \usepackage{amsmath, amsthm, amssymb}
\usepackage{amssymb}
\usepackage{graphicx}
\usepackage{algorithmic}
\usepackage{algorithm}

\usepackage{tikz}
\usetikzlibrary{calc}
\usetikzlibrary{patterns}
\tikzstyle{mybox} = [draw=black, fill=white,  thick,
    rectangle, inner sep=10pt, inner ysep=20pt]
\tikzstyle{mybox} = [draw=black, fill=white,  thick,
    rectangle, inner sep=2pt, inner ysep=2pt]

\usepackage[margin=1.2in]{geometry}

\usepackage{setspace}
\newtheorem{definition}{Definition}[section]
\newtheorem{lemma}{Lemma}[section]
\newtheorem{theorem}{Theorem}[section]
\newtheorem{conjecture}{Conjecture}[section]
\newtheorem{remark}{Remark}[section]

\begin{document}

\title{On the chromatic number of almost stable  general Kneser hypergraphs}
\author{Amir Jafari}
\maketitle

\begin{abstract}

 Let $n\ge 1$ and $s\ge 1$ be integers.  An almost $s$-stable subset $A$ of $[n]=\{1,\dots,n\}$ is a subset such that for any two distinct elements $i, j\in A$, one has $|i-j|\ge s$. For a family $\cal F$ of non-empty subsets of $[n]$ and an integer $r\ge 2$, the chromatic number of the $r$-uniform Kneser hypergraph $\mbox{KG}^r({\cal F})$, whose vertex set is  $\cal F$ and whose edge set is the set of $\{A_1,\dots, A_r\}$ of pairwise disjoint elements in $\cal F$, has been studied extensively in the literature and Abyazi Sani and Alishahi were able to give a lower bound for it in terms of the equatable $r$-colorability defect, $\mbox{ecd}^r({\cal F})$. In this article, the methods of Chen for the special family of all $k$-subsets of $[n]$, are modified to give lower bounds for the chromatic number of almost stable general Kneser hypergraph $\mbox{KG}^r({\cal F}_s)$ in terms of $\mbox{ecd}^s({\cal F})$. Here ${\cal F}_s$ is the collection of almost $s$-stable elements of $\cal F$. We also propose a generalization of a conjecture of Meunier.
 \end{abstract}

\section{Introduction}

Let $n\ge 1$, $s\ge 1$ and $r\ge 2$ be integers. Let $\cal F$ be a family of non-empty subsets of $[n]=\{1,\dots, n\}$. We say that a subset $A$ of $[n]$ is $s$-stable if for all distinct elements $i$ and $j$ in $A$, one has
$$s\le |i-j|\le n-s.$$
If we only demand $|i-j|\ge s$, then $A$ is said to be almost $s$-stable. We use the notation ${\cal F}_s$ for the almost $s$-stable subsets in $\cal F$. The $r$-uniform Kneser hypergraph $\mbox{KG}^r({\cal F}_s)$ is an $r$-uniform hypergraph whose vertex set is ${\cal F}_{s}$ and whose edge is the set of all pairwise disjoint subsets $\{A_1,\dots, A_r\}$ in ${\cal F}_s$.
We use the notion of the equitable $r$-colorability defect of $\cal F$, defined by Abyazi Sani and Alishahi in \cite{AA}. It is defined as follows.
\begin{definition}
The $r$-colorability defect of a family of non-empty subsets $\cal F$ of $[n]$ is defined to be the minimum size of a subset $X_0\subseteq [n]$ such that there is an equitable partition 
$$[n]\backslash X_0=X_1\cup\dots\cup X_r$$
so that there are no $F\in {\cal F}$ and $1\le i\le r$ such that $F\subseteq X_i$. Here equitable means that $||X_i|-|X_j||\le 1$ for all $1\le i\le j\le r$.
\end{definition}
Our goal here is to prove the following two theorems.
\begin{theorem}\label{thm1}
If $r$ is a power of $2$ and $s$ is a multiple of $r$, then 
\[\chi(\mbox{KG}^r({\cal F}_s))\ge \left\lceil \frac{\mbox{ecd}^s({\cal F})}{r-1}\right\rceil.\]
\end{theorem}
It is plausible to make the following conjecture.
\begin{conjecture}\label{conj}
For any $n\ge 1$, $r\ge 2$, $s\ge r$ and any family $\cal F$ of subsets of $[n]$, one has
\[\chi(\mbox{KG}^r({\cal F}_s))\ge \left\lceil \frac{\mbox{ecd}^s({\cal F})}{r-1}\right\rceil.\]
\end{conjecture}
\begin{remark} This conjecture for the special family of all $k$-subsets of $\{1,\dots, n\}$ was made by Meunier in \cite{M}. A version of this conjecture with the topological $r$-colorability defect was made for a general family and the $s$-stable part of the family by Frick in \cite{F}. 
\end{remark}
\noindent
We also prove the following theorem.
\begin{theorem}\label{thm2}
If $r=p$ is a prime number and $s\ge 2$ is an integer, then 

\[\chi(\mbox{KG}^p({\cal F}_{s}))\ge \left\lceil \frac{n-\alpha_1-\alpha_2}{p-1}\right\rceil\]
where $\alpha_1=(s-1)\left\lfloor \frac{n-\mbox{ecd}^p({\cal F})}{p}\right\rfloor$ and $\alpha_2=\left\lfloor (p-1)\frac{n-\mbox{ecd}^p({\cal F})+1}{p}\right\rfloor$.
\end{theorem}

\begin{remark}
If $p=2$ and $\cal F$ is the family of all $k$-subsets of $[n]$, then $\mbox{ecd}^2({\cal F})=n-2(k-1)$, and hence $\alpha_1=(s-1)(k-1)$ and $\alpha_2=k-1$. It follows that 
$$\chi(\mbox{KG}^2({\cal F}_s))\ge n-s(k-1)=\mbox{ecd}^s({\cal F}).$$
This gives a confirmation of the conjecture \ref{conj} for $r=2$ and the family of all $k$-subsets of $[n]$. This was proved by Chen in \cite{C2}. Also it is worthwhile to note that if $n\ge sk$, then coloring each almost $s$-stable $k$-subset with the value of its minimum element gives a proper coloring of $\mbox{KG}^2({\cal F}_s)$ with $n-s(k-1)$ colors. So, in fact the above inequality is an equality. 
\end{remark}

\section{Proof of Theorem \ref{thm1}}

The proof is done in two steps. First, we prove the theorem for $r=2$. Next, we prove that the statement of the theorem is true for all powers of $2$ by induction.
The first part of the proof is via a clever application of the Tucker lemma, which has its roots in the works of Meunier \cite{M} and Chen \cite{C1}. For $n\ge 1$, we let $sd(E_{n-1}(\mathbb Z_2))$ denote all non-empty subsets of $\{\pm 1,\dots, \pm n\}$ such that no two distinct elements of it, have the same absolute value.

\begin{proof} 
We will assume that $r=2$, $t=\chi(\mbox{KG}^2({\cal F}_s))$, with a proper coloring $c$ of its vertices with $\{1,\dots, t\}$. Let $\alpha=n-\mbox{ecd}^s({\cal F})$. Also, fix an arbitrary complete ordering on the subsets of $[n]$, such that if $|A|<|B|$ then $A<B$.
\\
We define a $\mathbb Z_2$-equivariant map
$$\lambda: sd(E_{n-1}(\mathbb Z_2)) \longrightarrow \{\pm 1,\dots, \pm m\}$$
where $m=\alpha+t$. For a non-empty face $A =\{a_1, a_2, \ldots , a_k\}$ with $|a_1| <|a_2| < \ldots < |a_k| $, we define ${\rm sgn (A)}$ equal to the sign of $a_1$. Also, we put $||A||=\{|a_1| , |a_2| , \ldots , |a_k| \}$. By an alternating subset of $A$, we mean a non-empty subset $\{a_{i_1} , a_{i_2} , \ldots , a_{i_l} \}\subseteq A$ with $|a_{i_1}| <|a_{i_2}| < \cdots < |a_{i_l}|$ in such a way that $a_{i_t}.a_{i_{t+1}} <0$ for each $1\leq t < l$. We define ${\rm Alt}(A)$ equal to the unique alternating subset $B\subseteq A$, such that for any other alternating subset $B' \subseteq A$, the relation $||B'||< ||B||$ holds. In other words, ${\rm Alt}(A)$ is the largest alternating subset of $A$ with respect to the complete ordering. 

Let $A =\{a_1, a_2, \ldots , a_k\}$ and ${\rm Alt}(A)= \{a_{i_1} , a_{i_2} , \ldots , a_{i_l} \}$. One may easily observe that $a_{1}.a_{i_1} >0$; since otherwise, $|a_{1}| < |a_{i_1}| $ and because of $a_{1}.a_{i_1} <0$, we obtain a larger alternating subset  $\{a_{1}, a_{i_1} , a_{i_2} , \ldots , a_{i_l} \}\subseteq A$; a contradiction.
Therefore, ${\rm sgn (A)}$ is exactly the sign of $a_{i_1}$.

\noindent Now, the definition of $\lambda (A)$ is given in two cases.

\noindent {\textbf{Case 1:}} If $|{\rm Alt}(A)| \leq \alpha$, then define $\lambda (A) = {\rm sgn}(A)|{\rm Alt}(A)|$.	 

\noindent {\textbf{Case 2:}} If $|\mbox{Alt}(A)|>\alpha$, then consider $\mbox{Alt}(A)=\{a_{i_1} , a_{i_2} , \ldots , a_{i_l} \}$ and define $$X_i=\{|a_{i_j}|\: : \: j\equiv i \mod s\}$$ for $i=1,\dots, s$. Then we have a family of $s$ equitable disjoint subsets such that $\sum_{i=1}^s |X_i|>n-\mbox{ecd}^s({\cal F})$. So, by the definition of equitable $s$-colorability defect, there are some $F\in {\cal F}$ and $1\le i\le s$, such that $F\subseteq X_i$. Choose the largest such $F$, using the chosen complete ordering. From the definition of $X_i$'s, it is clear that $X_i$'s are almost $s$-stable and hence $F$ is also almost $s$-stable, also since $s$ is even, the elements of $X_i$ are either inside $A^{+}$ or $A^{-}$, where $A^{\pm}=\{1\le i\le n | \pm i\in A\}$. Define $\lambda(A)=c(F)+\alpha$ if $F\subseteq A^+$ and $\lambda(A)=-(c(F)+\alpha)$ if $F\subseteq A^-$.

We now show that if $A =\{a_1, a_2, \ldots , a_k\}\subseteq B=\{b_1, b_2, \ldots , b_{k'}\}$ and $|\lambda(A)|=|\lambda(B)|$ then $\lambda(A)=\lambda(B)$. So, the Tucker lemma will imply that $m\ge n$ that is $t\ge \mbox{ecd}^s({\cal F})$ as desired. To prove this claim, we consider two parts: 

\noindent {\textbf{Part 1:}}
If $|\lambda (A)|=|\lambda (B)|\leq \alpha$, we are in the first case and hence $|{\rm Alt}(A)|=|{\rm Alt}(B)|$. Again, put $\mbox{Alt}(A)=\{a_{i_1} , a_{i_2} , \ldots , a_{i_l} \}$ and ${\rm Alt} (B)=\{b_{j_1},b_{j_2},\ldots ,b_{j_l}\}$ with  $|a_{i_1}| <|a_{i_2}| < \cdots < |a_{i_l}|$ and $|b_{j_1}|<|b_{j_2}|<\cdots <|b_{j_l}| $.
Since ${\rm sgn} (A) = {\rm sgn} (a_{i_1})$ and ${\rm sgn} (B) = {\rm sgn} (b_{1})$, if
${\rm sgn} (A) \neq {\rm sgn} (B)$, then ${\rm sgn} (b_{1}).{\rm sgn} (a_{i_1}) < 0$, and also, $|b_1| < |a_{i_1}|$. Therefore, we obtain an alternating subset $\{b_1,a_{i_1} , a_{i_2} , \ldots , a_{i_l} \}\subseteq A$ which is larger than ${\rm Alt}(A)$, which is impossible.
Consequently, $\lambda (A) = \lambda (B)$.

\noindent {\textbf{Part 2:}}
If $|\lambda (A)|=|\lambda (B)| > \alpha$, then we are in the second case and if $\lambda(A)=-\lambda(B)$, and say $\lambda(A)>0$, we have two subsets $F\subseteq A^{+}$ and $F'\subseteq B^{-}$ with the same color $c(F)=c(F')$, while $F\cap F'\subseteq A^{+}\cap B^{-}=\emptyset$, which is a contradiction to the properness of the coloring $c$.

So the claim and the theorem are proved for the case when $r=2$.
\end{proof}

Now we prove the following lemma, which will finish the proof of the Theorem \ref{thm1} by a simple induction on the exponent of $2$.
\begin{lemma}\label{lem}
 If Conjecture \ref{conj} is true for $(r,s)=(r_1,r_1)$ and for $(r,s)=(r_2,s_2)$, then it will be true for $(r,s)=(r_1r_2,r_1s_2)$.
 \end{lemma}

\begin{proof} Let $r=r_1r_2$ and $s=r_1s_2$. Let $t=\chi(\mbox{KG}^r({\cal F}_s))$ and $c:{\cal F}_s\rightarrow \{1,\dots, t\}$ be a proper coloring. Define
$${\cal F}'=\{X\subseteq [n]| \mbox{ecd}^{s_2}({\cal F}|_X)>(r_2-1)t\}.$$
 Let $X\in {\cal F}'_{r_1}$, then by identifying $X$ with $\{1,2,\dots, |X|\}$ by the unique order preserving bijection, any element of $({\cal F}|_X)_{s_2}$ will be in ${\cal F}_s$ and hence gets a color in $\{1,\dots, t\}$ via $c$. Since by the assumption of the lemma

$$\chi(\mbox{KG}^{r_2}(({\cal F}|_X)_{s_2}))>t$$
therefore, one can find pairwise disjoint subsets $B_1(X),\dots, B_{r_2}(X)$ in ${\cal F}_s|_X$ with the same color, we assign that color to $X$ and hence get a coloring $c':{\cal F}'_{r_1}\rightarrow \{1,\dots, t\}$. We claim that $c'$ is a proper coloring of $\mbox{KG}^{r_1}({\cal F}'_{r_1})$ and hence by the hypothesis of the lemma
$$\mbox{ecd}^{r_1}({\cal F}')\le (r_1-1)t.$$
To prove the claim, assume for the sake of contradiction that we have found pairwise disjoint subsets $A_1,\dots, A_{r_1}$ in ${\cal F}'_{r_1}$ with the same color $c'(A_1)=\dots=c'(A_{r_1})$, then it follows that we have $r=r_1r_2$ pairwise disjoint subsets $B_i(X_j)$ for $i=1,\dots, r_2$ and $j=1,\dots, r_1$ in ${\cal F}_{s}$ of the same color by the coloring $c$. This contradicts the properness of $c$. Now, we may find $X_0\subseteq [n]$ of size at most $(r_1-1)t$ and an equitable partition
$$[n]\backslash X_0=X_1\cup \dots \cup X_{r_1}$$
such that no $X\in {\cal F}'$ is a subset of one of $X_i$'s for $i=1,\dots, r_1$. In particular $X_i\not\in {\cal F}'$. So $\mbox{ecd}^{s_2}({\cal F}|_{X_i})\le (r_2-1)t$. So one may find $X_{i,0}\subseteq X_i$ of size at most $(r_2-1)t$ and an equitable partition,
$$X_i\backslash X_{i,0}=X_{i,1}\cup \dots \cup X_{i,s_2}$$
such that no $F\in {\cal F}|_{X_i}$ is a subset of one of $X_{i,1},\dots, X_{i, s_2}$. Note that we may assume without loss of generality that $|X_{i,0}|=(r_2-1)t$, by removing elements from $X_{i,j}$'s and adding them to $X_{i,0}$ without violating the equitability condition. This implies that $X_{i,j}$ for $i=1,\dots, r_1$ and $j=1,\dots, s_2$ form an equitable partition of $[n]\backslash X_0'$ with $X_0'=X_0\cup X_{1,0}\cup\dots \cup X_{r_1,0}$ with at most $(r_1-1)t+r_1(r_2-1)t=(r-1)t$ elements. This partition has the property that there is no $F\in {\cal F}$ such that it is a subset of one of $X_{i,j}$ for $1\le i\le r_1$ and $1\le j\le s_2$. So, $\mbox{ecd}^s({\cal F})\le (r-1)t$ and the lemma is proved.
\end{proof}

\section{Proof of Theorem \ref{thm2}}

This is also by modifying  a proof of Chen given in \cite{C2}. It uses the $\mathbb Z_p$-Tucker lemma whose statement is given in \cite{M}. For $n\ge 1$ and a prime number $p$, we let 
$sd(E_{n-1}(\mathbb Z_p))$ denote all non-empty subsets of $\mathbb Z_p\times [n]$, such that no two distinct elements of it, have the same second component.

\begin{proof}
Let $t=\chi(\mbox{KG}^p({\cal F}_s))$ and $c$ be a proper coloring of its vertices with $\{1,\dots, t\}$.
 We construct a $\mathbb Z_p$-equivariant map 
$$\lambda: sd(E_{n-1}(\mathbb Z_p))\longrightarrow \mathbb Z_p\times [m]$$
$$\lambda(A)=(\lambda_1(A),\lambda_2(A))$$
where $m=t+\alpha_1+\alpha_2$ and satisfies the properties of the $\mathbb Z_p$-Tucker lemma with parameters $\alpha=\alpha_1+\alpha_2$ and $m$, i.e.
\begin{enumerate}
\item If $A_1\subseteq A_2$ and $\lambda_2(A_1)=\lambda_2(A_2)\le \alpha$ then $\lambda_1(A_1)=\lambda_1(A_2)$.
\item If $A_1\subseteq \dots \subseteq A_p$ and $\lambda_2(A_1)=\dots=\lambda_2(A_p)>\alpha$ then $\lambda_1(A_1),\dots, \lambda_1(A_p)$ are not pairwise distinct.
\end{enumerate}
 and hence
$$\alpha+(p-1)(m-\alpha)\ge n.$$
Therefore
$$t\ge \frac{n-\alpha}{p-1}$$
and the theorem is proved. Let $A$ be a non-empty face and $B\subset A$ be such that for each $i\in \mathbb Z_p$, $B^i=\{1\le j\le n| (i,j)\in B\}$ is an almost $s$-stable subset and $\pi_2(B)\subseteq [n]$ is maximum with respect to an arbitrary complete ordering on subsets of $[n]$, with the property that if $|A|<|B|$ then $A<B$. Here $\pi_2$ is the projection onto the second component.  The construction of $\lambda$ is given in three cases. 
\\
{\textbf{Case 1:}}  If there is an $F\in{\cal F}$ with $F\subseteq B^i$ for some $i\in {\mathbb Z}_p$, then choose the smallest such $F$ and define
$$\lambda(A)=(i, c(F)+\alpha).$$
\\
{\textbf{Case 2:}} If $|B^{i_1}|=\dots=|B^{i_p}|$, and $1\le j\le s-1$ is so that the smallest element of $\pi_2(A)$ is congruent to $j$ modulo $s-1$, then define
$$\lambda(A)=(sgn(A), (s-1)(|B^{i_1}|-1)+j).$$
Note that since we are not in the case one,  $p|B^{i_1}|\le n-\mbox{ecd}^p({\cal F})$ and hence $$(s-1)(|B^{i_1}|-1)+j\le \alpha_1.$$
\\
{\textbf{Case 3:}} Otherwise, If $|B^{i_1}|=\dots=|B^{i_h}|<|B^{i_{h+1}}|\le \dots\le |B^{i_p}|$ with some $1\le h<p$, then choose $1\le h'<p$ such that $hh'\equiv 1 \mod p$ and define
$$\lambda(A)=((i_1\dots i_h)^{h'}, (p-1)|B^{i_1}|+p-h+\alpha_1).$$
Note that since we are not in the case one, if we remove elements from $B^{i_{h+1}},\dots, B^{i_p}$ so that their size become $|B^{i_1}|+1$, then we have an equitable disjoint collection of sets without any element of $\cal F$ in them, so $p|B^{i_1}|+p-h\le n-\mbox{ecd}^p({\cal F})$ and hence it follows that
$$(p-1)|B^{i_1}|+p-h\le \alpha_2.$$

It remains to check the conditions of $\mathbb Z_p$-Tucker lemma. The equivariance is only non-obvious for the case 3, it follows from the fact that 
$$((\omega\cdot i_1)\dots (\omega\cdot i_h))^{h'}=\omega^{hh'}\cdot (i_1\dots i_h)^{h'}=\omega\cdot  (i_1\dots i_h)^{h'}.$$
If $A_1\subseteq \dots\subseteq A_p$ and $\lambda_2(A_1)=\dots=\lambda_2(A_p)>\alpha$, then we are in the first case and we have elements $F_i\in{\cal F}$ such that $F_i\subseteq B_i^{\lambda_1(A_i)}$ and have the same color $c(F_1)=\dots=c(F_p)$. Hence if $\lambda_1(A_1),\dots, \lambda_1(A_p)$ are pairwise distinct then $F_1,\dots, F_p$ are pairwise disjoint. This contradicts the properness of $c$.
\\
If $A_1\subseteq A_2$ and $\lambda_2(A_1)=\lambda_2(A_2)\le \alpha_1$, then we are in the second case, hence $|B_1^i|=|B_2^i|$ for $i\in \mathbb Z_p$. So if $sgn(A_1)\ne sgn(A_2)$, then since the first elements of $\pi_2(A_1)$ and $\pi_2(A_2)$ are congruent modulo $s-1$, by adding the first element of $A_2$, to $B_2$ we get subset $B_3$, where $B_3^i$ is almost $s$-stable for all $i$ and $B_3^{sgn(A_1)}$ is bigger than $B_1^{sgn(A_1)}$. This contradicts the maximality of $B_1$.
\\
If $A_1\subseteq A_2$ and $\alpha_1<\lambda_2(A_1)=\lambda_2(A_2)\le \alpha$, then we are in the third case. Hence if
$$|B_1^{i_1}|=\dots=|B_1^{i_h}|<|B_1^{i_{h+1}}|\le \dots\le |B_1^{i_p}|$$
$$|B_2^{j_1}|=\dots=|B_2^{j_k}|<|B_2^{j_{k+1}}|\le \dots\le |B_2^{j_p}|$$
then $|B_1^{i_1}|=|B_2^{j_1}|$ and $h=k$ and therefore $\{i_1,\dots, i_h\}=\{j_1,\dots, j_k\}$. It follows that $\lambda_1(A_1)=\lambda_1(A_2)$. 
\\
All of the conditions are checked and hence the theorem is proved.
\end{proof}

\section{Conclusion}

Let $n\ge 1$ and $s\ge 2$, $r\ge 2$ be integers. A subset $A$ of $[n]$ is $s$-stable if $s\le |i-j|\le n-s$ for all distinct $i,j\in A$. For a family $\cal F$ of subsets in $[n]$, the chromatic number of
$\mbox{KG}^r({\cal F}_{s\tiny{\mbox{-stab}}})$ for the family ${\cal F}_{s\tiny{\mbox{-stab}}}$ of the $s$-stable elements in $\cal F$ has also been studied in the literature. Similar to conjecture \ref{conj}, one might be tempted to make the following conjecture.
But it has counter examples even for the case $r=s=2$, see \cite{F}.

\begin{conjecture}\label{conj2}
With the above notation for $s\ge r$, one has
\[\chi(\mbox{KG}^r({\cal F}_{s\tiny{\mbox{-stab}}}))\ge \left\lceil \frac{\mbox{ecd}^s({\cal F})}{r-1}\right\rceil.\]
\end{conjecture}
However it is true for $r=s=2$ , for the family of all $k$-subsets of $[n]$, which is just the statement of the famous Schrijver's theorem \cite{S}. Also in \cite{C1}, Chen proves the Conjecture \ref{conj2} for the family of all $k$-subsets of $[n]$, $r=2$ and $s$ an even integer. 
\\
{\textbf{Acknowledgement.}} The author wishes to thank Hamidreza Daneshpajouh and Saeed Shaebani, for stimulating discussions and sharing their ideas.

\end{document}